\documentclass[11pt]{amsart}

\usepackage{amsmath, amsthm, amssymb, amscd, cancel, graphicx, pinlabel, soul, stmaryrd}

 \usepackage{url}
 \usepackage{graphicx}
 \usepackage[all]{xy}
 \usepackage{epsfig}
\usepackage[margin=1.3in]{geometry}
\usepackage{enumerate}
\usepackage[hyperfootnotes=false]{hyperref}

\newtheorem{theorem}{Theorem}[section]

\newtheorem{claim}[theorem]{Claim}
\newtheorem{conjecture}[theorem]{Conjecture}
\newtheorem{corollary}[theorem]{Corollary}
\newtheorem{lemma}[theorem]{Lemma}
\newtheorem{proposition}[theorem]{Proposition}
\newtheorem*{problem*}{Problem}
\newtheorem{question}{Question}
\newtheorem*{question*}{Question}
\newtheorem{mainthm}{Theorem}

\makeatletter
\newtheorem*{rep@theorem}{\rep@title}
\newcommand{\newreptheorem}[2]{%
\newenvironment{rep#1}[1]{%
 \def\rep@title{#2 \ref{##1}}%
 \begin{rep@theorem}}%
 {\end{rep@theorem}}}
\makeatother

\newreptheorem{theorem}{Theorem}
\newreptheorem{lemma}{Lemma}

\newtheorem{definition}[theorem]{Definition}

\theoremstyle{remark}
\newtheorem*{remark}{Remark}

\theoremstyle{plain} 
\newcommand{\thistheoremname}{}
  \newtheorem*{genericthm*}{\thistheoremname}
\newenvironment{namedthm*}[1]
  {\renewcommand{\thistheoremname}{#1}%
   \begin{genericthm*}}
  {\end{genericthm*}}

\newcommand\simtimes{\mathbin{%
    \stackrel{\sim}{\smash{\times}\rule{0pt}{0.6ex}}%
    }}

\newcommand{\negativecrossing}
{\raisebox{0.12in}
{\includegraphics[scale=0.04, angle =270]{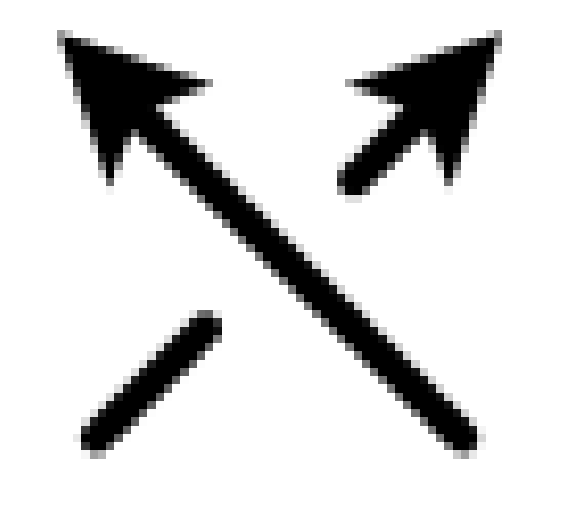}}}

\newcommand{\one}
{\raisebox{-0.04in}
{\includegraphics[scale=0.35]{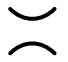}}}

\newcommand{\zero}
{\raisebox{-0.04in}
{\includegraphics[scale=0.35]{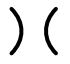}}}

\usepackage[usenames,dvipsnames,svgnames,table]{xcolor}

\hypersetup{
  colorlinks,
  citecolor=RoyalPurple,
   linkcolor=RoyalPurple,
   urlcolor=RoyalPurple}

\title{On Amphichirality of Symmetric Unions}
\date{}
\author{Fer\.{I}de Ceren K\"ose}
\address{Department of Mathematics, University of Texas at Austin}

\email{fkose@math.utexas.edu}

\begin{document}





\begin{abstract}
It is still unknown whether there is a nontrivial knot with Jones polynomial equal to that of the unknot. Tanaka shows that if an amphichiral knot is a symmetric union of the unknot with one twist region, then its Jones polynomial is trivial. Hence, he proposes that if any of these knots were nontrivial, a nontrivial knot with trivial Jones polynomial would exist. We first show such a knot is always trivial and hence cannot be used to answer the above question. We then generalize the argument to symmetric unions of any knots and show that if a symmetric union of a knot $J$ with one twist region is amphichiral, then it is the connected sum of $J$ and its mirror image $-J$. 
\end{abstract}

\maketitle
\section{Introduction}
Whenever a knot invariant is defined, a fundamental question arises: Are there distinct knots that share the same invariant? In particular, does it distinguish nontrivial knots from the unknot? For the Alexander polynomial — the first knot polynomial — the answers to these questions are known. The pretzel knots $P(3,2k,-3)$ have the same Alexander polynomial for every integer $k$ and all knots in the Kinoshita-Terasaka family have trivial Alexander polynomial. In fact, among the Kanenobu knots there is an infinite family all of which are fibered and have the same Alexander module \cite{Ka81}. In \cite{Ka86} Kanenobu further extends the study of his knots and provides infinite knot families with the same HOMFLY polynomial and hence the same specializations of HOMFLY polynomial (Alexander polynomial, Jones polynomial, etc). Although, it has been known for many years that the families above, along with a plethora of other families, answer both questions for the Alexander polynomial and the first question for the Jones polynomial, the question of whether a nontrivial knot with trivial Jones polynomial exists remains open.

\begin{problem*}[Kirby Problem 1.88(C)]
Is there a nontrivial knot with the same Jones polynomial as the unknot?
\end{problem*}

Knot Floer homology \cite{OS04} and Khovanov homology \cite{Kh00} categorify the Alexander and Jones polynomials. Recent developments in these theories now verify the second question: Knot Floer homology detects the unknot, and in fact, the knot genus \cite{OS04}. Khovanov homology also detects the unknot \cite{KM11}. Yet still, they fail to distinguish many other knots. The pretzel knots $P(3,2k,-3)$ have identical knot Floer homology \cite{HW18} and some infinite families of Kanenobu knots have identical Khovanov homology \cite{Wa07}. 

It is astonishing that the above problem remains unanswered despite the fact that it was answered for the Alexander polynomial decades ago and has even been answered for more recent knot homologies. A few unsuccesful attempts were made to find such a knot in \cite{APR89} and \cite{JR94} by applying transformations on a complicated diagram of the unknot which do not change the Jones polynomial but possibly change the knot type. Computer searches by \cite{DH97}, \cite{TS18} and \cite{TS21} showed that this knot must have a crossing number of at least 24. 

The problem of finding such a knot is also related to a famous open question in braid groups that asks whether the Burau representation of $B_4$ is faithful. It was a widely believed conjecture that if unfaithful, such a knot exists. The conjecture was recently resolved by Ito who proved that the unfaithfulness of the Burau representation of $B_4$ would indeed imply the existence of such a knot \cite{It15}. On the other hand, such a knot, if it exists, would disprove the volume conjecture \cite{murakamis:volume}.

The same question regarding links with more than one component, however, does have an answer; in \cite{EKT03} Eliahou, Kauffman and Thistlewaite exhibit infinite families of prime $k$-component links with Jones polynomial equal to that of the $k$-component unlink for each $k \geq 2$. 

In the context of studying these questions on knot invariants all of the families mentioned above belong to a class of ribbon knots, called \textit{symmetric unions}. A symmetric union of a knot is a classical construction introduced by Kinoshita and Terasaka \cite{KT57}. A detailed construction and examples are given in Section \ref{S2}. Kinoshita-Terasaka knots are symmetric unions of the unknot, Kanenobu knots are symmetric unions of the Figure-8 knot, and the pretzel knots $P(3,2k,-3)$ are symmetric unions of the trefoil. In fact, one can produce more knots with trivial or the same Alexander polynomial using this construction as the Alexander polynomial of a symmetric union of a knot $J$ depends only on the Alexander polynomial of $J$ and the parity of the number of inserted crossings in each twist region (\cite{KT57}, \cite{La00}). 

Thus, if there really exists a nontrivial knot with trivial Jones polynomial, then it is quite natural to expect it to be a symmetric union. To that end, let us consider the following formula for the Jones Polynomial of a symmetric union with one twist region in terms of the Jones Polynomial of $J$ and the number of inserted crossings given by Tanaka. 

\begin{theorem}[\cite{Ta15}]{\label{ThmTanaka}}
Let $K$ be a knot that admits a symmetric union presentation of the form $(D_J \cup -D_J)(n)$, then
\begin{align*}
    t^n V_K(t)+(-1)^n V_K(t^{-1}) = (t^n + (-1)^n)V_J(t)V_J(t^{-1}). 
\end{align*}
In particular, if $K$ is amphichiral, then $V_K(t) = V_J(t)V_J(t^{-1})$.
\end{theorem}

Let $U$ denote the unknot. Then in light of Theorem \ref{ThmTanaka}, Tanaka proposes the following.

\begin{proposition}[\cite{Ta15}, Proposition 7.1]{\label{Prop1}}
If there exists a nontrivial, amphichiral knot that admits a symmetric union presentation of the form $(D_U \cup -D_U)(n)$, then its Jones polynomial is trivial and, in particular, the answer to the above problem is affirmative.
\end{proposition}

We give two proofs that the knots of the form stated in Proposition \ref{Prop1} are in fact trivial and hence cannot be used to answer the above problem in the affirmative. In the first proof, we show that the double branched cover of such a knot corresponds to a pair of cosmetic surgeries on a connected sum of knots in $S^3$. We then use the surgery characterization of the unknot in $S^3$ and the fact that composite knots do not admit cosmetic surgeries, a result of Tao \cite{Tao19}, to conclude such a pair does not exist. The second proof uses Khovanov homology's unknot detection result and the Lee spectral sequence to obstruct the amphichirality of these knots.

\begin{mainthm}{\label{ThmA}}
Let $K$ admit a symmetric union presentation of the form $(D_U \cup -D_U)(n)$. Then, $K$ is amphichiral if and only if $K$ is the unknot.
\end{mainthm}

We observe that all known symmetric union diagrams of prime amphichiral ribbon knots have more than one twist region \cite{La00} ,\cite{La06}, \cite{La17}. Hence, we question whether an analogue of the same result holds for any partial knot. In this case, we are able to generalize the topological proof of Theorem \ref{ThmA}; we show that the double branched cover corresponds to a pair of cosmetic surgeries in a $\mathbb{Q}$-homology sphere and analyzing the JSJ-decompositions of those surgeries we conclude that they can never be homeomorphic. Combining with Theorem \ref{ThmA}, we obtain the following. This is the main theorem of the paper.

\begin{mainthm}{\label{ThmB}}
Let $K$ admit a symmetric union presentation of the form $(D_J \cup -D_J)(n)$ for a knot $J$. Then, $K$ is amphichiral if and only if $K = J \# -J$.
\end{mainthm}

Another application of our results regards the following fundamental question in the study of symmetric unions.

\begin{question}[\cite{La06}]{\label{Q1}}
Does every ribbon knot admit a symmetric union presentation?
\end{question}

All but 15 prime ribbon knots with 12 and fewer crossings \cite{La00}, \cite{Se14} and all 2-bridge knots \cite{La06} have been shown to be symmetric unions. Among those 15 ribbon knots, $12a990$ and $12a1225$ are amphichiral. Thus, it follows from our result that if any of them admits a symmetric union diagram, then such a diagram has at least two twist regions. This is indeed the case for all nontrivial amphichiral ribbon kots and the following corollary is immediate.

\begin{corollary}{\label{cor1}}
Let $K$ be a nontrivial prime amphichiral ribbon knot. If $K$ admits a symmetric union presentation of the form $(D \cup -D)(n_1, \dots , n_k)$, then $k \geq 2$.
\end{corollary}

The knot $10_{42}$ was shown to be the first ribbon knot that do not admit a symmetric union diagram with one twist region by an ad hoc computation of the Jones polynomial by Tanaka \cite{Ta15}. The first ribbon knots with a structural property that obstructs them from admitting such a diagram are composite ribbon knots which are neither $K\#-K$ for a nontrivial knot $K$, nor  $K_1\#K_2$ where each summand $K_i$ is a symmetric union, such as $3_1\#8_{10}, 3_1\#8_{11}$ etc. \cite{kearney:concordance}, \cite{livingston:concordance}. This was also proven by Tanaka \cite{Ta19} and later a short proof was given by the author using the ideas in this paper \cite{Ko20}. Theorem \ref{ThmB}, therefore, gives the first infinite family of prime ribbon knots with a structural property – being amphichiral – that obstructs them from admitting the simplest type of symmetric union presentations.

\subsection*{Organization}
In Section \ref{S2}, we define symmetric unions, give examples, necessary definitions and a Dehn filling description of the double branched cover. In Section \ref{S3}, we prove Theorem \ref{ThmA}. In Section \ref{S4}, we prove Theorem \ref{ThmB} in the case where $J$ is nontrivial.

\subsection*{Acknowledgement}
The author is indebted to her advisor Cameron Gordon for his constant encouragement and many insightful discussions. The author thanks Robert Lipshitz and Sucharit Sarkar for explaining why Kronheimer and Mrowka's unknot detection result can also be stated in terms of the Khovanov homology with coefficients in $\mathbb{Q}$. The author is especially grateful to Robert Lipshitz for his interest in this work and helpful correspondence. The author is also grateful to Christoph Lamm, whose beautiful pictures of symmetric unions in the Appendix of \cite{La17} helped her greatly to make the observation which now became Theorem \ref{ThmB}.

\section{Definitions and the set up}{\label{S2}}
\subsection{Symmetric unions}
Consider $S^3$ identified as $\mathbb{R}^3 \cup \{\infty\}$ where $\{\infty\}$ denotes the point at infinity. We assume a knot diagram is depicted in the $xy$-plane. Let $\mathcal{S}$ denote the sphere $\{x=0\} \cup \{\infty\}$ and $\tau$ be the reflection through $\mathcal{S}$ which is an orientation-reversing involution on $S^3$. We call the balls bounded by $\mathcal{S}$ $\textit{the left ball}$ $B_L$ and $\textit{the right ball}$ $B_R$ according to the natural position of their projections in the $xy$-plane. We define a symmetric union as follows.

\begin{figure}[ht]{\label{tanglereplacement}}
    \centering
    \includegraphics[width=13cm]{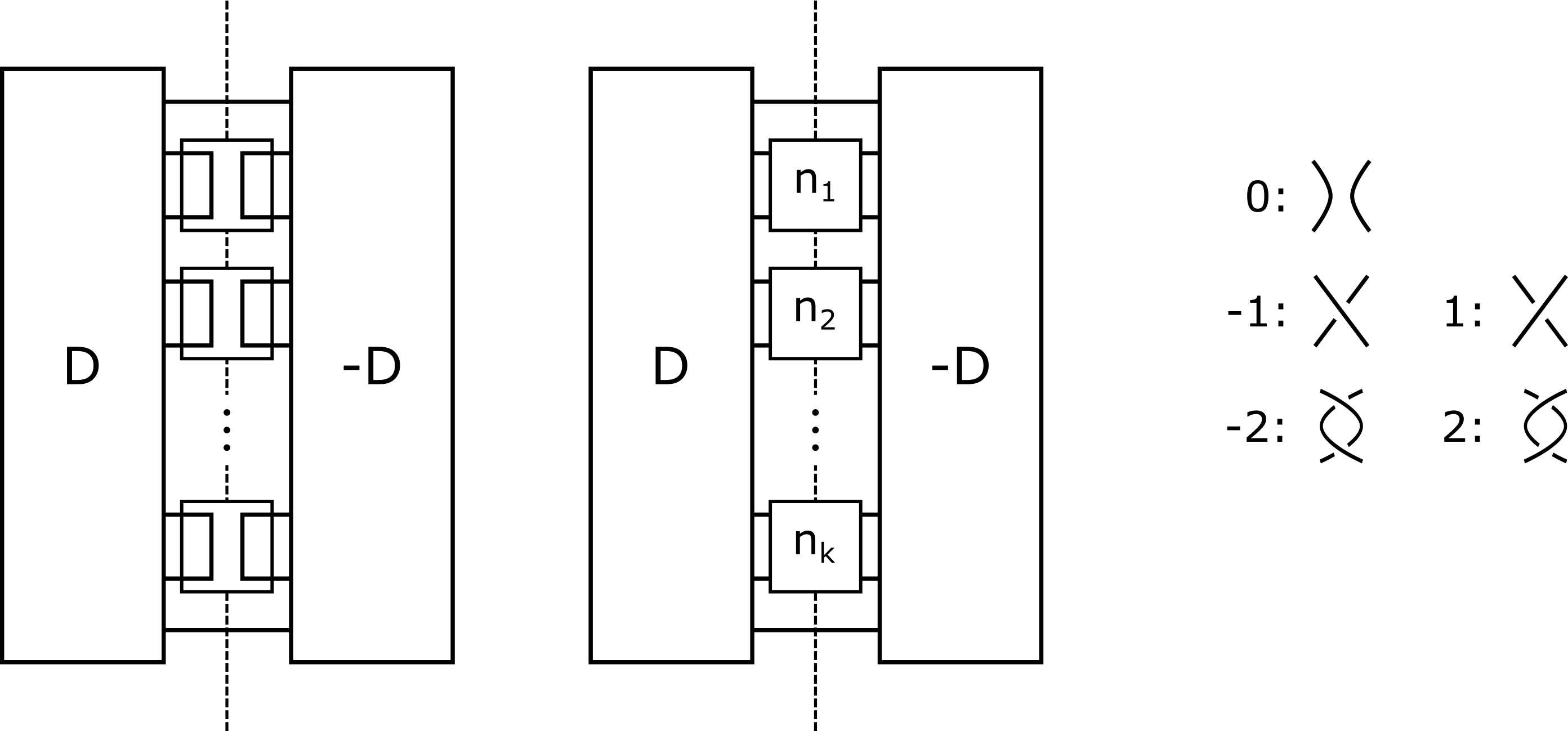}
    \caption{The construction of a symmetric union}
    \label{figure:construction}
    \end{figure}

\begin{definition}
Let $D$ be a knot diagram in $B_L$ and $B_1, \dots , B_k$ be the balls along the y-axis such that each of which intersects $D$ in a trivial arc. Let $-D$ be the diagram obtained by reflecting $D$ across $\mathcal{S}$. We take the connected sum of $D$ and $-D$ such that the resulting diagram $D\#-D$ remains invariant under $\tau$ and hence $D\#-D$ intersects each $B_i$ in a trivial tangle. A symmetric union of $D$ is then defined to be a knot diagram obtained from $D \# -D$ by replacing the trivial tangle in each $B_i$ by $\frac{1}{n_i}$-tangle for some $n_i \in \mathbb{Z}$ and is denoted by $(D \cup -D)(n_1,\dots ,n_k)$. See Figure \ref{figure:construction} for the schematic description.
\end{definition}

The construction was first introduced by Kinoshita and Terasaka \cite{KT57} in which there is a single tangle replacement, then was generalized to multiple tangle replacements by Lamm \cite{La00}. The y-axis, drawn as the dashed line in Figure 1, is called \textit{the axis of mirror symmetry}. A knot which admits such a diagram is called a $\textit{symmetric union}$. The isotopy type of $D$ is called the \textit{partial knot} of $(D\cup -D)(n_1, \dots, n_k)$. For a symmetric union $K$ and its partial knot $J$, we may choose to denote $K$ by $(D_J \cup -D_J)(n_1, \dots,n_k)$ to specify $J$ and call $K$ a symmetric union of $J$. The balls $B_1, \dots, B_k$ are called the twist regions and the number of nonzero elements in $\{n_1, \dots ,n_k\}$ is called the number of twist regions of $(D \cup -D)(n_1, \dots ,n_k)$. See Figure \ref{figure:symmetricunions} for examples of symmetric unions.

\begin{figure}
    \centering
    \includegraphics[height=3cm]{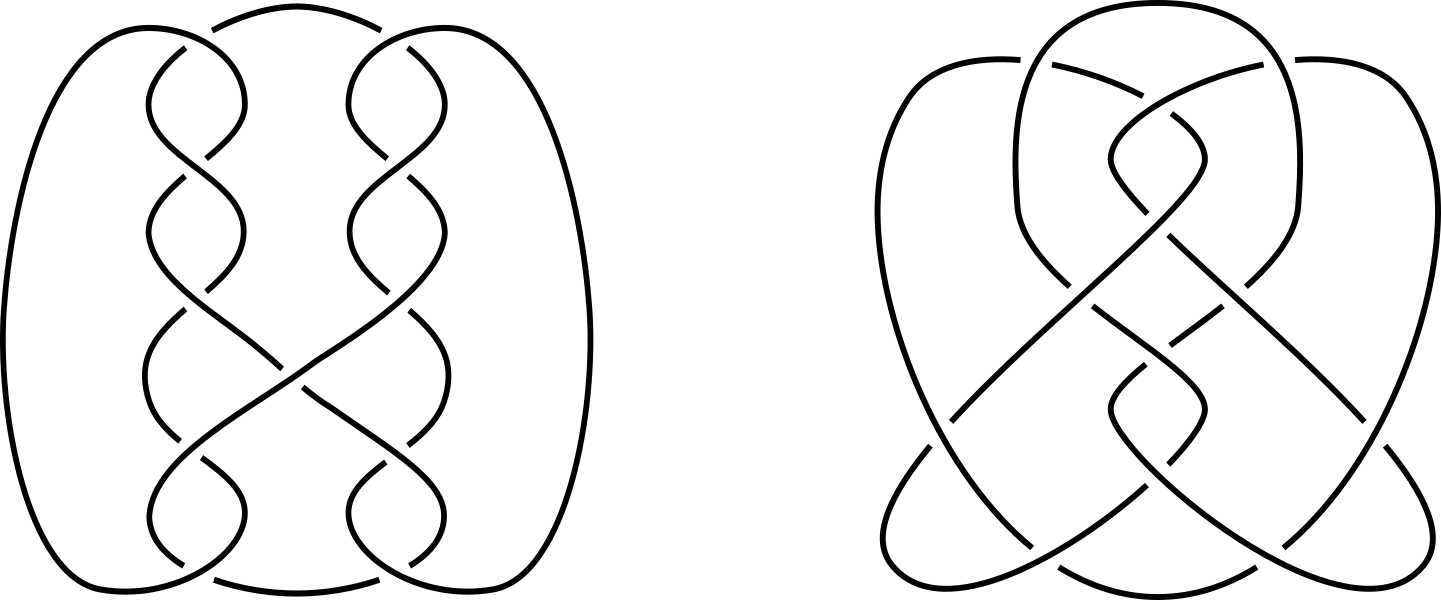}
    \caption{On the left is the knot $10_{153}$ presented as $(D_U \cup -D_U)(1)$ and on the right is the knot 12n462, which is in the family of Kanenobu knots, presented as $(D_{4_1} \cup -D_{4_1})(2,-2)$.}
    \label{figure:symmetricunions}
    \end{figure}

\subsection{Double branched covers and Dehn fillings}
Let $\Sigma(K)$ denote the double cover of $S^3$ branched along the knot $K$ and $\Sigma(B,A)$ denote the double cover of a ball $B$ branched along a tangle $A$. Let $M$ be a compact, connected, orientable 3-manifold and $K$ be a knot in $M$. The \textit{exterior} of $K$ in $M$ is $M_K := M \setminus \text{int } \nu (K)$, where $\nu (K)$ is a regular neighborhood of $K$. Recall that \textit{a slope} on $T^2$ is the unoriented isotopy class of a non-trivial simple closed curve that is a primitive class $\in H_1(T^2)/ \pm 1$. Since $H_1(T^2) \cong \mathbb{Z} \oplus \mathbb{Z}$, the slopes on $T^2$ may be parameterized by reduced rational numbers $\frac{p}{q} \in \mathbb{Q} \cup \{ \frac{1}{0} \}$ once a basis is fixed for $H_1(T^2)$. We then define the $\frac{p}{q}$-\textit{Dehn surgery} on $K$ in $M$, $M_K(\frac{p}{q})$, to be the manifold obtained by gluing a solid torus $V$ to $M_K$ so that the boundary of a meridional discs of $V$ is glued to $\frac{p}{q}$: $M_K(\frac{p}{q}) = M_K \cup V$. Note that when $M \cong S^3$, we make an exception and stick to the commonly used notation $S^3_{p / q}(K)$. It is natural to extend the notion of Dehn surgery on a knot to that of Dehn filling on a 3-manifold along some torus boundary component. So let $M$ be a 3-manifold with torus boundary, and suppose that $\frac{p}{q}$ is a slope on $\partial M$. Similarly, we define the \textit{$\frac{p}{q}$-Dehn filling} on $M$, $M(\frac{p}{q})$, to be the manifold obtained by gluing a solid torus $V$ to $M$ so that the boundary of a meridional disc of $V$ is glued to $\frac{p}{q}$.

In this article, we focus on knots that admit symmetric union diagrams with one twist region, i.e. the crossings are inserted in only one region. So suppose $K$ is such a symmetric union with a partial knot $J$ and let $(D \cup -D)(n)$ represent $K$ for some nonzero $n$. The relevant set up is already described in \cite{Ko20}. We repeat it here for convenience. Let $B$ denote the ball containing the twist region for the inserted crossings. $B$ lifts to a solid torus $\tilde{B}$ in the double branched cover $\Sigma_2(J\#-J)$. Let $\overline{K}$ denote the core of $\tilde{B}$. As $K$ is obtained from $J \# -J$ replacing a $\frac{1}{0}$-tangle in $B$ by a $\frac{1}{n}$-tangle, using the Montesinos trick \cite{Mo75} we see that the tangle replacement corresponds to a $\frac{1}{n}$-Dehn surgery along $\overline{K}$ giving us a description of $\Sigma(K)$ in terms of $\Sigma(J)$: $\Sigma(K) = \Sigma(J \# -J)_{\overline{K}}(\frac{1}{n})$. Let $X$ denote the ball $S^3 \setminus \text{int } {B}$ and $\tilde{X}$ denote the double branched cover of $X$ branched along the tangle $t = K \cap  X$.  One can easily see that $\tilde{X}$ is $\Sigma(J \# -J)_{\overline{K}}$ and hence, $\Sigma(K)$ is $\frac{1}{n}$-Dehn filling of $\tilde{X}$, $\tilde{X}(\frac{1}{n})$. This description relies on viewing $B$ as the pillow case. Let $(\mu, \lambda)$ be the canonical meridian-longitude pair induced by $B$ where $\mu$ is a meridian of $\overline{K}$ determined by some choice of an orientation on $\overline{K}$.

\section{The case where the partial knot is the unknot}{\label{S3}}

In this section, we prove Theorem \ref{ThmA} and hence conclude that the knots proposed by Tanaka cannot be used to find nontrivial knots with trivial Jones polynomial. We present two proofs of Theorem \ref{ThmA}. The first proof, given in Subsection \ref{S3.1}, uses a surgery argument made in the double branched cover and the second proof, given in Subsection \ref{S3.2}, uses Khovanov homology. In Subsection \ref{S3.2} we also ask a question about symmetric union diagrams of amphichiral ribbon knots based on an observation and discuss how Khovanov homology can be used to answer this question.

\subsection{Cosmetic Surgeries in $S^3$}{\label{S3.1}}
All 3-manifolds in this article are assumed to be compact and oriented. $-M$ will denote the manifold $M$ with opposite orientation. For given 3-manifolds $M$ and $M^\prime$, we write $M \cong M^\prime$ if there exists an orientation-preserving homeomorphism $h: M \rightarrow M^\prime$ and we say $M$ and $M^\prime$ are homeomorphic as oriented manifolds. If $h$ is orientation-reversing, then it will be explicitly stated. 

Two surgeries $S^3_r(K)$ and $S^3_s(K)$ on a knot $K$ in $S^3$ are called $\textit{cosmetic}$ if they are homeomorphic as unoriented manifolds, and \textit{purely cosmetic} if they are homeomorphic as oriented manifolds. Cosmetic surgeries that are not purely cosmetic are called $\textit{chirally cosmetic.}$ It is not hard to find chirally cosmetic surgeries: $S^3_r(K)$ and $S^3_{-r}(K)$ are orientation-reversingly homeomorphic if $K$ is amphichiral. However, there are no known examples of purely cosmetic surgeries to date and in fact it was conjectured by Gordon \cite{Go91} that there are no purely cosmetic surgeries on a nontrivial knot.

\begin{conjecture}[The Purely Cosmetic Surgery Conjecture in $S^3$]{\label{ConjCos}}
Let $K$ be a nontrivial knot in $S^3$. If $r\neq s$, then $S^3_r(K) \ncong S^3_s(K)$.
\end{conjecture}

Conjecture \ref{ConjCos} is verified for 2-bridge knots and alternating fibered knots by \cite{IJMS19}, for knots of Seifert genus one by \cite{Wa06}, for cable knots by \cite{TaoCable}, for connected sums of non-trivial knots by \cite{Tao19}, for 3-braid knots by \cite{Va21}, for prime knots with at most 16 crossings by \cite{Ha20}, and for pretzel knots by \cite{SS20}. By Moser's classification of Dehn surgeries on torus knots \cite{Mo71} and the classification of Seifert fibered spaces, Conjecture \ref{ConjCos} also holds for torus knots. As will be seen below, it is relevant to us to record the theorem that proves the conjecture for composite knots.

\begin{theorem}[\cite{Tao19}] {\label{ThmTao}}
Let $K$ be a composite knot in $S^3$. If $r\neq s$, then $S^3_r(K) \ncong S^3_s(K)$.
\end{theorem}

Now suppose $K$ is a symmetric union of the form $(D_J \cup -D_J)(n)$. Recall that $\Sigma(K) = \Sigma(J\#-J)_{\overline{K}}(\frac{1}{n})$. The lift of $\mathcal{S}$, $\tilde{\mathcal{S}}$, decomposes $\Sigma(J \# -J)$ as $\Sigma(J) \# \Sigma(-J)$. Because $\tau$ leaves $J \# -J$ invariant, it lifts to an orientation-reversing involution $\tilde{\tau}$ on $\Sigma(J \# -J)$ whose fixed points set is $\tilde{\mathcal{S}}$. Let $\Sigma(J)_o$ and $\Sigma(-J)_o$ denote the punctured $\Sigma(J)$ and $\Sigma(-J)$, respectively, that are obtained by cutting $\Sigma(J\#-J)$ along $\tilde{\mathcal{S}}$. Hence, we observe the following.

\begin{lemma}{\label{L4}}
$\overline{K}$ decomposes as $K^\prime \# -K^\prime$ where $K^\prime = \Sigma(J)_o \cap \overline{K}$ and $-K^\prime = \Sigma(-J)_o \cap \overline{K}$. 
\end{lemma}

\begin{proof}
This is immediate by the construction as $\overline{K} \cap \tilde{S}$ is exactly two points and $\overline{K}$ is left invariant by $\tilde{\tau}$.
\end{proof}

With Lemma \ref{L4} at hand, Theorem \ref{ThmA} follows from the characterization of the unknot by its double branched cover, a famous theorem of Waldhause \cite{Wa69}, and Theorem \ref{ThmTao}.

\begin{proof}[Proof of Theorem \ref{ThmA} ]
The backward direction is obvious. Now suppose $K$ is amphichiral, then $\Sigma(K) \cong \Sigma(-K)$ as oriented manifolds. Since $J$ is the unknot, $\Sigma(J\#-J)$ is $S^3$. Then $\Sigma(K) = S^3_{1/n}(\overline{K})$ and $\Sigma(-K) = S^3_{-1/n}(\overline{K})$. Thus, we obtain a pair of purely cosmetic surgeries $S^3_{1/n}(\overline{K}) \cong S^3_{-1/n}(\overline{K})$ unless $\overline{K}$ is trivial. By Lemma \ref{L4} and Theorem \ref{ThmTao}, we conclude that $\overline{K}$ is indeed trivial. Therefore; $S^3_{1/n}(\overline{K}) \cong S^3_{-1/n}(\overline{K}) \cong S^3$ and by \cite{Wa69} $K$ is the unknot.

\end{proof}

\subsection{Khovanov Homology}{\label{S3.2}}
Theorem \ref{ThmA} now confirms that symmetric union diagrams of prime amphichiral ribbon knots have indeed more than one twist region. In addition to this, as seen in the appendix of \cite{La00} the numbers of inserted crossings in the symmetric union diagrams found for all prime amphichiral ribbon knots but 12a435 come in pairs with opposite signs. For instance, 12a477 is of the form $(D_{6_3} \cup - D_{6_3})(1,-1)$ and 12a458 is of the form $(D_{8_3} \cup -D_{8_3})(1,-1,1,-1,1,-1)$ etc. This is also the case for the Kanenobu knots: Recall that the Kanenobu knots $K_{p,q}$ are of the form $(D_{4_1} \cup -D_{4_1})(p,q)$ and it is shown in \cite{Ka86} that $K_{p,q}$ is amphichiral if and only if $p+q=0$. Suppose $K$ is an amphichiral ribbon knot that admits a symmetric union presentation, then we ask the following.

\begin{question}{\label{Q2}}
Does $K$ admit a symmetric union presentation $(D \cup -D)(n_1, \dots ,n_k)$ such that $k$ is even and the set $\{n_1 , \dots , n_k\}$ is composed of pairs of the form $(n,-n)$?
\end{question}

We wonder whether Khovanov homology may give information about the inserted crossings in symmetric union diagrams and can be used to answer Question \ref{Q2}. To provide evidence we present another proof of Theorem \ref{ThmA} by using Khovanov homology. We show that the inserted crossings induce bigradings in which the Khovanov homology is nonzero and use those to obstruct the amphichirality of these knots.

We begin reviewing some facts from Khovanov homology that will be used in the proof. For a link $L \subset S^3$, $Kh(L)$ will denote the Khovanov homology of $L$ with coefficients in $\mathbb{Q}$ and we use the notation $Kh^{i,j}(L)$ to specify the homology with homological grading $i$ and quantum grading $j$. 
Given a knot $K$ with a distinguished negative crossing, let $(\negativecrossing)$ denote a diagram of $K$, and $(\one)$ and $(\zero)$ denote the diagrams obtained by resolving $(\negativecrossing)$ in the two possible ways. Observe that $(\one)$ inherits an orientation from $(\negativecrossing)$, but for $(\zero)$ there is no orientation consistent with $(\negativecrossing)$. Then, as $(\zero)$ is a knot we may orient it as we please and get that
\begin{align*}
     c =   \text{number of negative crossings in $(\zero)$}   - \text{number of negative crossings in $(\negativecrossing)$}
\end{align*}
is well-defined. This long exact sequence 
\begin{equation}{\label{LES}}
\longrightarrow Kh^{i,j+1}(\one)\longrightarrow Kh^{i,j}(\negativecrossing) \longrightarrow Kh^{i-c,j-3c-1}(\zero) \longrightarrow Kh^{i+1,j+1}(\one) \longrightarrow 
 \end{equation}
 
 is one of the Khovanov's original results in \cite{Kh00}, but our notational convention follow Turner \cite{Tu17}.

 In \cite{Lee05}, Lee showed that $Kh(L)$ can be viewed as the first page of a spectral sequence, the Lee spectral sequence, that converges to the direct sum of $|L|$ copies of $\mathbb{Q} \oplus \mathbb{Q}$ with generators located in homological grading 0, where $|L|$ is the number of components of $L$. In particular, when $L$ is a knot, Rasmussen \cite{Ras10} added to this result that the quantum gradings of the generators of $E_\infty = \mathbb{Q} \oplus \mathbb{Q}$ always differ by 2, leading him define his $s$-invariant by taking the average of these two quantum gradings. The $s$-invariant provides a lower bound on the slice genus of a knot, $g_s(K)$; hence, it vanishes for slice knots; hence, for symmetric unions. Then we conclude that the two surviving generators of the Lee spectral sequence for the Khovanov homology of symmetric unions are located at bigradings $(0,-1)$ and $(0,1)$. It is also relevant for us to note that the differential of this spectral sequence on the $E_n$ page is of bidegree $(1,4n)$.
 
Another key ingredient of our argument is the detection of the unknot. Let $\overline{Kh}(K)$ denote the reduced Khovanov homology of $K$ over $\mathbb{Q}$.

\begin{theorem}[\cite{KM11}]
A knot $K$ in $S^3$ is the unknot if and only if $\overline{Kh}(K) \cong \mathbb{Q}$. 
\end{theorem}

Even though the theorem is stated in terms of reduced Khovanov homology, from the way the reduced Khovanov complex is defined from the Khovanov complex (see Section 3 in \cite{Kh00} for more details), one can see that a similar result holds for Khovanov homology. For that purpose, suppose $Kh(K)$ is 2-dimensional. Then it consists only of the Lee generators, which are located at gradings $(0,s-1)$ and $(0,s+1)$. Hence, the Khovanov complex of $K$ decomposes as a direct sum of an acyclic complex and the complex $0 \rightarrow \mathcal{A} \rightarrow 0$, where $\mathcal{A} = \mathbb{Q}[x]/(x)$. Tensoring this complex over $\mathcal{A}$ with $\mathbb{Q}[x]/(x) = \mathbb{Q}$ yields a complex with 1-dimensional homology. Thus by \cite{KM11} we obtain \\

\begin{corollary}{\label{Cor1}}
A knot $K$ in $S^3$ is the unknot if and only if $Kh(K)$ is 2-dimensional.
\end{corollary}
    
     \begin{figure}[ht]{\label{FigKh}}
    \centering
    \includegraphics[height=3.3cm]{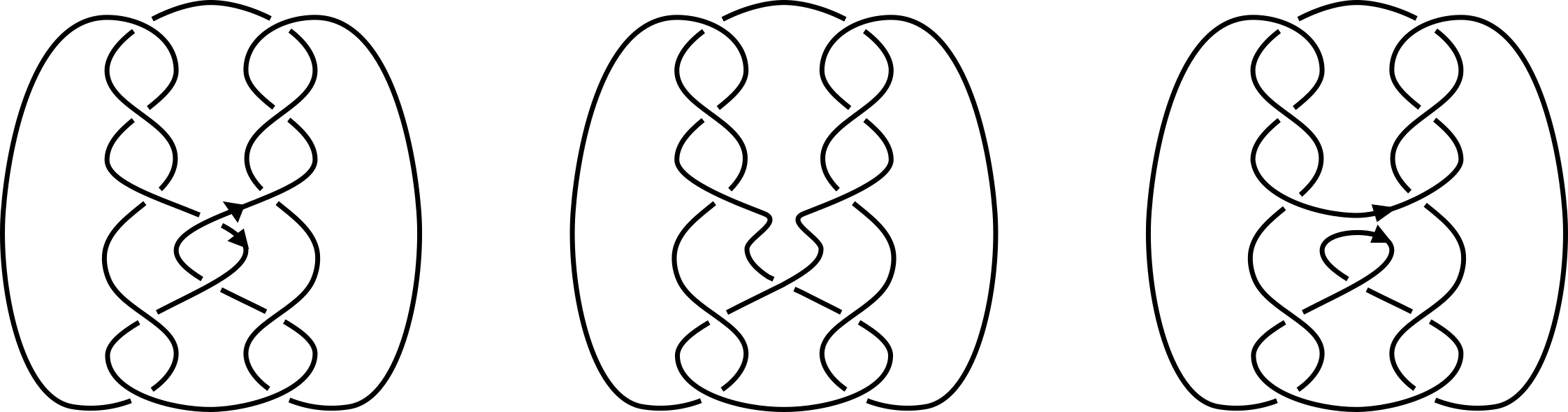}
    \caption{The resolutions of a crossing of a symmetric union diagram of the Kinoshita-Terasaka knot along the axis of mirror symmetry.}
    \label{resolutions}
    \end{figure}

The sequence of coefficients of the Jones polynomial of an amphichiral link is palindromic. Similar to this, the Khovanov homology of an amphichiral link $L$ satisfies the symmetric relation $Kh^{i,j}(L) = Kh^{-i,-j}(L)$. Hence, Theorem \ref{ThmA} follows from the lemma below. Note that since the diagram $(D \cup -D)(-n_1,\dots ,-n_k)$ describes the mirror image of $(D \cup -D)(n_1,\dots ,n_k)$, we will only prove the lemma for positive integers. 

\begin{lemma}{\label{lemma:khovanov}}
For a given diagram $D_U$ of the unknot, let $K_n$ be the symmetric union $(D_U \cup -D_U)(n)$ where $n \in \mathbb{N}$. Suppose there exists $n$ such that $K_n$ is non-trivial and then let $k$ be the minimal positive integer such that $K_k$ is nontrivial. Then, for every nonnegative integer $m$ there exists a bigrading $(i,j)$ such that $Kh^{i,j}(K_{k+m}) \not\cong Kh^{-i,-j}(K_{k+m})$.
\end{lemma}

\begin{proof}
Any crossing of $K_n$ along the axis of symmetry gives rise to a triple $\{(\negativecrossing), (\zero),(\one) \}$, where $K_n = (\negativecrossing)$ and $K_{n-1} = (\zero)$. The resolution $(\one)$ is a two-component link $L$ and $L$ is amphichiral as it is left invariant by $\tau$. An example of such a triple is demonstrated in Figure \ref{resolutions}. Setting $n =k$, by the minimality of $k$ we obtain that $K_{k-1}$ is the unknot. For symmetric union diagrams, the number of positive and negative crossings away from the axis are the same and equal to the number of crossings of the diagram of its partial knot. When $n >0$, the inserted crossings are negative. Therefore, the difference $c$ in negative crossings between $(\zero)$ and $(\negativecrossing)$ is -1 and the triple gives rise to the following long exact sequence
\begin{equation}{\label{LES1}}
 \overset{\delta_*}{\longrightarrow} Kh^{i,j+1}(L)\longrightarrow Kh^{i,j}(K_k) \longrightarrow Kh^{i+1,j+2}(U) \overset{\delta_*}{\longrightarrow} .
 \end{equation}

The unknot has Khovanov homology $\mathbb{Q}_{(-1)} \oplus \mathbb{Q}_{(1)}$ supported in homological grading 0, therefore; whenever $i \neq -1,0$ or $j \neq -3, -1$ the sequence becomes
\begin{equation}{\label{Liso}}
0 \longrightarrow Kh^{i,j+1}(L)\longrightarrow Kh^{i,j}(K_k) \longrightarrow 0,
\end{equation}
and we obtain the isomorphisms $Kh^{i,j}(K_k) \cong Kh^{i,j+1}(L)$. By Corollary \ref{Cor1} $Kh(K_k)$ has at least one generator that is different than the Lee generators. Then let $(a,b)$ be the bigrading of a non Lee generator of $Kh(K_k)$ such that $b$ is the maximal quantum grading and $a$ is the maximal homological grading among those bigradings in quantum grading $b$. For convenience, for a non-trivial link we will call the bigrading that satisfies the above condition \textit{maximal}.

\begin{claim}{\label{L3}}
$b \geq 1$. Furthermore; if $b = 1$, then $a \neq 0$.
\end{claim}

\begin{proof}
Assume $b < 1$. Since the bigrading $(a,b)$ does not survive to the $E_\infty$ page and $b$ is the maximal quantum grading, $Kh^{a-1, b-4m}(K_k) \neq 0$ for some $m \in \mathbb{N}$. Because $b - 4m \leq -5$ and $-b+4m-2 \geq 3$, the isomorphism (\ref{Liso}) and $L$'s being amphichiral imply
\begin{align*}
  0 \neq  Kh^{a-1,b-4m}(K_k) \cong Kh^{a-1,b-4m+1}(L) \cong Kh^{-a+1,-b+4m-1}(L) \cong Kh^{-a+1,-b+4m-2}(K_k),
\end{align*}
but $-b + 4m -2 > 1$ which contradicts the maximality of $b$.

For the second assertion, we assume $(a,b) = (0,1)$. Then $rank$ $ Kh^{0,1}(K_k) \geq 2$. Using this with the isomorphism (\ref{Liso}) and $L$'s being amphichiral, we conclude that $rank$ $Kh^{0,-2}(L) \geq 2$. When $(i,j)=(0,-3)$, the long exact sequence (\ref{LES1}) becomes
\begin{align*}
& \longrightarrow \mathbb{Q} \longrightarrow Kh^{0,-2}(L) \longrightarrow Kh^{0,-3}(K_k)\longrightarrow  0,
\end{align*}
via exactness $Kh^{0,-3}(K_k)\neq 0$. Because this bigrading does not survive to the $E_\infty$ page, either $Kh^{1,-3+4m}(K_k)$ or $Kh^{-1,-3-4m}(K_k)$ is nonzero for some $m \in \mathbb{N}$. As $(0,1)$ is the maximal bigrading, we have $Kh^{-1,-3-4m}(K_k) \neq 0$. As $-3-4m \leq -7$ and  $2+4m \geq 6$, the isomorphism (\ref{Liso}) and the fact that $L$ is amphichiral imply
\begin{align*}
  0 \neq  Kh^{-1,-3-4m}(K_k) \cong Kh^{-1,-2-4m}(L) \cong Kh^{1,2+4m}(L) \cong Kh^{1,1+4m}(K_k),
\end{align*}
a contradiction.
\end{proof}

Note that as $b \geq 1$, by the isomorphisms (\ref{Liso}) $Kh^{a,b}(K_k) \cong Kh^{a,b+1}(L)$, and hence, $(a,b+1)$ is the maximal bigrading for $Kh(L)$.

\begin{claim}{\label{C1}}
$(a,b)$ is the maximal grading of $Kh(K_{k+m})$ for all $m \in \mathbb{Z}_{\geq 0}$.
\end{claim}

\begin{proof}
We proceed by induction on $m$. The base case, $m=0$, is immediate as $(a,b)$ is the maximal bigrading for $Kh(K_k)$. Resolving any crossing of $K_{k+m}$ along the axis of symmetry results in an unoriented triple $\{(\negativecrossing), (\zero),(\one) \}$, where $K_{k+m} = (\negativecrossing)$, $K_{k+m-1} = (\zero)$, and $L = (\one)$. From the induction hypothesis and $b+1$'s being the maximal quantum grading for $Kh(L)$, whenever $j > b$ the long exact sequence corresponding to this skein triple becomes
\begin{align*}
& 0 \longrightarrow Kh^{i,j}(K_{k+m})\longrightarrow  0,
\end{align*}
implying $Kh^{i,j}(K_{k+m}) = 0$, and when $(i,j)=(a,b)$ the sequence becomes
\begin{align*}
& 0 \longrightarrow Kh^{a,b+1}(L) \longrightarrow Kh^{a,b}(K_{k+m})\longrightarrow  0,
\end{align*}
implying the isomorphim $Kh^{a,b}(K_{k+m})= Kh^{a,b+1}(L) \neq 0$. 

\end{proof}

\begin{claim}{\label{C2}}
$Kh^{-a-m,-b-2(m+1)}(K_{k+m}) \neq 0$ for all $m \geq 0$.
\end{claim}

\begin{proof}
We also prove this by induction on $m$. By Lemma \ref{L3} $(-a,-b)\neq (0,-1)$ or $(0,1)$, then the long exact sequence (\ref{LES1}) induces an injection $Kh^{-a,-b-1}(L) \subseteq Kh^{-a,-b-2}(K_k)$. Using this together with the facts that $b>0$ and $L$ is amphichiral, we obtain the base case as follows
\begin{align*}
   0 \neq Kh^{a,b}(K) \cong Kh^{a,b+1}(L) \cong Kh^{-a,-b-1}(L) \subseteq Kh^{-a,-b-2}(K_k).
\end{align*}
Because $-b-1$ is the minimal quantum grading of $Kh(L)$, when $m>0$ the long exact sequence corresponding to the skein triple ${(K_{k+m},K_{k+m-1},L)}$ becomes
\begin{align*}
    & 0 \longrightarrow Kh^{-a-m,-b-2(m+1)}(K_{k+m}) \longrightarrow Kh^{-a-m+1,-b-2m}(K_{k+m-1})\longrightarrow  0,
\end{align*}
implying the isomorphism $Kh^{-a-m,-b-2(m+1)}(K_{k+m}) \cong Kh^{-a-m+1,-b-2m}(K_{k+m-1})$. Thus the claim follows from this isomorphism and the inductive hypothesis. 

\end{proof}

Now the lemma follows from Claim \ref{C1} and Claim \ref{C2} by setting $i=a+m$ and $j=b+2(m+1)$.
\end{proof}

\begin{remark}
The proof of Lemma \ref{lemma:khovanov} extends to any partial knot $J$ and shows that $Kh^{i,j}(K_k) =0$ for $j \geq b + 2k$ and $Kh^{-a-k,-b-2k}(K_k) \neq 0$ where $(a,b)$ is the maximal bigrading of $Kh(J\#-J)$ provided that $q_{max}(L) < b-1$. For instance, if $L$ is the 2-component unlink as in the cases of the 3-stranded pretzel knots $P(3,k,-3)$ and the Kanenobu knots, Theorem \ref{ThmB} is achieved. However, this condition is not always satisfied: The knot $12n605$ admits a symmetric union presentation of the form $(D_{3_1} \cup -D_{3_1})(1)$ where $q_{max}(L)=6$ while $q_{max}(3_1 \# -3_1)=7$; see the Appendix in \cite{La17}.
\end{remark}

\section{The case where the partial knot is not the unknot}{\label{S4}}
In this section, we will prove Theorem \ref{ThmB} when $J$ is a nontrivial knot. Then, together with Theorem \ref{ThmA} we will complete the proof of Theorem \ref{ThmB}.

\begin{theorem}{\label{ThmC}}
Let $K$ admit a symmetric union presentation of the form $(D_J \cup -D_J)(n)$ for a nontrivial knot $J$. Then, $K$ is amphichiral if and only if $K = J \# -J$. 
\end{theorem}

The backward direction is obvious. We will prove the contrapositive of the forward direction. Thus, we assume $K$ is of the form $(D_J \cup -D_J)(n)$ but is not $J\#-J$, the form stated in Theorem \ref{ThmC}. Let us begin by setting some foundation. By \cite{La00} we know that if $J$ is nontrivial, then $K$ is nontrivial. To show that, Lamm and Eisermann establish a surjection $\overline{\pi}(K) \rightarrow \overline{\pi}(J)$, where $\overline{\pi}(K)=\pi(K)/(m^2)$ is the $\pi-$orbifold group of $K$ which is obtained from the knot group with the additional relation that a meridian has order 2. It holds in general that for a knot $K$ there is a canonical surjection $\overline{\pi}(K) \rightarrow \mathbb{Z}_2$ resulting from the abelianization. Hence, the result is achieved by these two surjections and the fact that $\overline{\pi}(K)= \mathbb{Z}_2$ if and only if $K$ is the unknot, which follows by the proof of the Smith conjecture; see \cite[Proposition~3.2]{BZ89}.

If $K$ is composite, then $K = K_1 \# -K_1 \# K_2$ where $K_1$ is a nontrivial knot and $K_2$ is a symmetric union with one twist region (\cite{Ta19},\cite{Ko20}). Repeating this, we may assume that $K_2$ is prime and the partial knot $J$ contains $K_1$ as a connected summand. Then, because $K$ is not $J\#-J$, $K$ is either prime or has a prime summand which is a symmetric union with one twist region. In the latter case, if $K$ is amphichiral, then by the Unique Factorization Theorem of Schubert \cite{Sch49} this prime symmetric union is also amphichiral. Thus, it suffices to prove the statement assuming $K$ is prime. As shown in \cite{Ko20} if $K$ is prime, $\tilde{X}$ is irreducible, i.e. contains no essential sphere. Hence, for the remainder of the paper, $K$ is nontrivial and prime, and hence, $\Sigma(K)$ is not homeomorphic to $S^3$ and $\tilde{X}$ is irreducible. 

The outline of the proof of Theorem \ref{ThmC} is then as follows: If $K$ is amphichiral, i.e. $K\simeq -K$, then $\Sigma(K) \cong \Sigma(-K)$. By the Dehn filling description of $\Sigma(K)$ given in Section \ref{S2} and the fact that as $K$ is of the form $(D_J \cup -D_J)(n)$, then $-K$ is of the form $(D_J \cup -D_J)(-n)$, we obtain $\tilde{X}(\frac{1}{n}) \cong \tilde{X}(-\frac{1}{n})$. Analyzing the JSJ-decomposition of the fillings we show that they can never be homeomorphic.

\subsection{The JSJ-decomposition of double branched covers}{\label{S4.1}} In this subsection, we first review JSJ-decompositions and then give details about the JSJ-decompositions of $\tilde{X}$ and $ \tilde{X}(\frac{1}{n})$. 

We use the notation $S(g,n:(\alpha_1,\beta_1),\dots , (\alpha_k,\beta_k))$ to denote the Seifert fibered space with $n$ boundary components, $k$ fibers with coefficients $(\alpha_i,\beta_i)$ for $i=1,\dots , k$, and a base orbifold of genus $g$. A negative $g$ means that the base orbifold is non-orientable and $|g|$ is the first Betti number of this non-orientable surface. We may assume $\alpha_i >0$ as the pair $(\alpha_i, \beta_i)$ is only determined up to sign and the fiber with coefficient $(\alpha_i, \beta_i)$ is said to be \textit{singular} if $\alpha_i >1$. We would like to point out the Seifert fibered spaces $S(0,n)$ with $n \geq 3$ and $S(0,2;(\alpha,\beta))$, which are called $\textit{a composing space}$ and $\textit{a cable space}$ respectively, for being essential in our argument.

The following is the celebrated JSJ-decomposition theorem, which is independently due to \cite{JS78} and \cite{Jo79}.

\begin{theorem}[The JSJ-decomposition theorem] 
Let $M$ be a compact irreducible orientable 3-manifold with a boundary possibly empty or disjoint union of tori. Then there exists a finite collection of disjoint incompressible embedded tori $\{T_i\}$ in $M$ such that each component of M $\setminus \cup_i T_i$ is either atoroidal or Seifert fibered. Furthermore, any such collection with a minimal number of components is unique up to isotopy.
\end{theorem}

\begin{definition}
We will refer to such a minimal collection of tori $\{T_i\}$ as the $\textit{JSJ-tori}$ of $M$. Then, an embedded torus $T$ will be called a $\textit{JSJ-torus}$ if it is isotopic to a torus in the collection of the JSJ-tori. The components of $M$ cut along the JSJ-tori will be referred to as the $\textit{JSJ-pieces}$ of $M$.
\end{definition}

To decide whether a given collection of tori is JSJ or not, we will appeal to the following proposition:

\begin{proposition}\cite{AFW15}{\label{PropAFW}}
Let M be a compact irreducible orientable 3-manifold with empty or toroidal boundary. Let $\{T_i\}$ be a collection of disjoint embedded incompressible tori in M. Then $\{T_i\}$ are the JSJ-tori of M if and only if the following holds:

\begin{enumerate}[i.]
     \item each component $\{M_j\}$ of M $\setminus \cup_i T_i$ is atoroidal or Seifert fibered,
    \item if $T_i$ cobounds Seifert fibered components $M_j$ and $M_k$ (possibly with $j=k$), then their regular fibers do not match,
    \item if a component $M_j$ is homeomorphic to $T^2 \times I$, M is a torus bundle with only one JSJ-piece.
\end{enumerate}
\end{proposition}

\begin{definition}
From the JSJ-decomposition of a manifold $M$, there arises a graph in a natural way such that the vertices correspond to the JSJ-pieces and the edges correspond to the JSJ-tori along which adjacent JSJ-pieces are glued. We refer to this graph as the $\textit{JSJ-graph}$ of $M$.
\end{definition}

\begin{remark}
Because an embedded non-separating torus gives rise to an element of infinite order in homology, the JSJ-graph of a $\mathbb{Q}HS^3$ is acyclic. As we deal with double branched covers of knots, which are $\mathbb{Q}HS^3$, we simply refer to their JSJ-graphs as \textit{JSJ-trees}.
\end{remark}

\begin{figure}[ht]
    \centering
    \includegraphics[height=4cm]{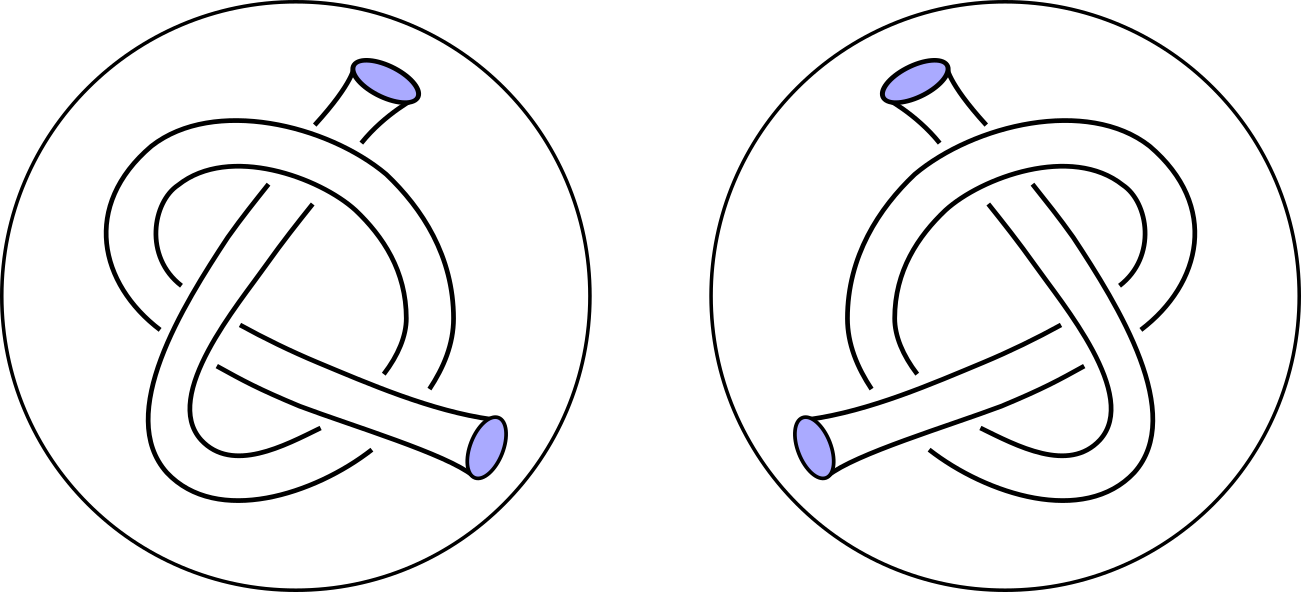}
    \caption{Schematic pictures of $Y$ and $-Y$. $\tilde{X}$ is obtained by gluing them along the annuli that are the complements of two purple disks.}
    \label{figure:complementofcompositeknots}
    \end{figure}
    
We now describe a composing space in $\tilde{X}$ arising as a result of Lemma \ref{L4}. Let $Y$ denote the knot complement $\Sigma(J)\setminus \nu (K^\prime)$ (similarly $-Y$ denotes $\Sigma(-J)\setminus \nu (-K^\prime)$ ). Consider a tubular neighborhood $\nu(\partial Y)$ in $ Y$. This neighborhood has two disjoint tori as boundary, one of which is $\partial Y$ and the other is interior in $Y$. Let $T_1$ denote this interior torus boundary and note that it bounds a 3-manifold homeomorphic to $Y$. Hence, let us abuse the notation and call this 3-manifold $Y$. We repeat the same construction for $-Y$ and denote the interior torus by $T_2$. A choice of orientations on $K^\prime$ and $-K^\prime$ determines meridians in $\partial Y$ and $\partial(-Y)$, respectively. Similarly, for $T_1$ and $T_2$ as well. We make a choice such that when we glue $\nu(\partial Y)$ and $\nu(\partial (-Y))$ along annuli, that are tubular neighborhoods of meridians in $\partial Y$ and $\partial (-Y)$, the orientations are compatible. Let $C$ denote this space obtained by gluing $\nu(\partial Y)$ and $\nu(\partial (-Y))$ and then we get $\tilde{X} = Y \cup_{T_1} C \cup_{T_2} -Y$. Let $C_n$ denote the space obtained by $ \frac{1}{n}$-Dehn filling of $C$. Then, from the construction we see that

\begin{proposition}{\label{PropFiber}}
$C$ is a composing space $S(0,3)$ whose regular fibers have meridional slopes on $\partial \tilde{X}$, $T_1$, and $T_2$.
Then $C_n$ is a cable space $S(0,2;(n,1))$ $(S(0,2;(n,-1))$ if $n<0)$ such that its regular fibers are also meridional on $T_1$ and $T_2$.
\end{proposition}

\begin{proof}
Because $\nu(\partial Y) \cong T^2 \times I$, we may choose to fiber $\nu(\partial Y)$ so that it is $S(0,2)$ and its regular fibers have meridional slopes on both $\partial Y$ and $T_1$. We fiber $\nu(\partial(-Y))$ in a similar way. Then, as $C$ is obtained by gluing $\nu(\partial Y)$ and $\nu(\partial (-Y))$ along annuli, that are tubular neighborhoods of meridians in $\partial Y$ and $\partial (-Y)$, the regular fibers of $\nu(\partial Y)$ and $\nu(\partial(-Y))$ agree under gluing and hence the fibrations extend over $C$. The desired result now follows immediately as the base orbifold is the boundary sum of two annuli, which is a pair of pants. The second assertion also follows immediately as the fiber is meridional on $\partial \tilde{X}$. 

\end{proof}

In light of Proposition \ref{PropFiber}, the JSJ-decomposition of $\tilde{X}(\frac{1}{n})$ is induced by the JSJ-decomposition of $\tilde{X}$. Appealing to Proposition \ref{PropAFW}, the JSJ-decomposition of $\tilde{X}$ can be determined from those of $Y$ and $-Y$ when $T_1$ or $T_2$ are incompressible in $\tilde{X}$. By the symmetry $\tilde{X}$ governs, standard innermost disk arguments, and the fact that the components of $\partial C$ are incompressible in $C$, $T_1$ or $T_2$'s being incompressible in $\tilde{X}$ is the same as $\partial Y$'s being incompressible in $Y$. As any essential sphere in $Y$ would give rise to an essential sphere in $\tilde{X}$, we conclude that $Y$ is also irreducible. Hence, if $\partial Y$ is compressible in $Y$, then $Y$ is a solid torus. In this case, $\tilde{X}$ and $\tilde{X}(\frac{1}{n})$ are Seifert fibered spaces and we analyze this case by studying homeomorphisms between Seifert fibered spaces. When $Y$ is not a solid torus, we then appeal to the following lemma to have a description of the JSJ-decompositions of $\tilde{X}$ and $\tilde{X}(\frac{1}{n})$ in terms of those of $Y$ and $-Y$. Before we state the lemma, note when $|n|=1$ $C_{\pm 1}, \cong T^2 \times I$. As $\tilde{X}(\pm 1)$ is not a torus bundle, $C_{\pm 1}$ is not a JSJ-piece and hence $|n|=1$ is considered as a separate case.

   \begin{figure}[ht]{\label{JSJtrees}}
    \centering
    \includegraphics[height=4.5cm]{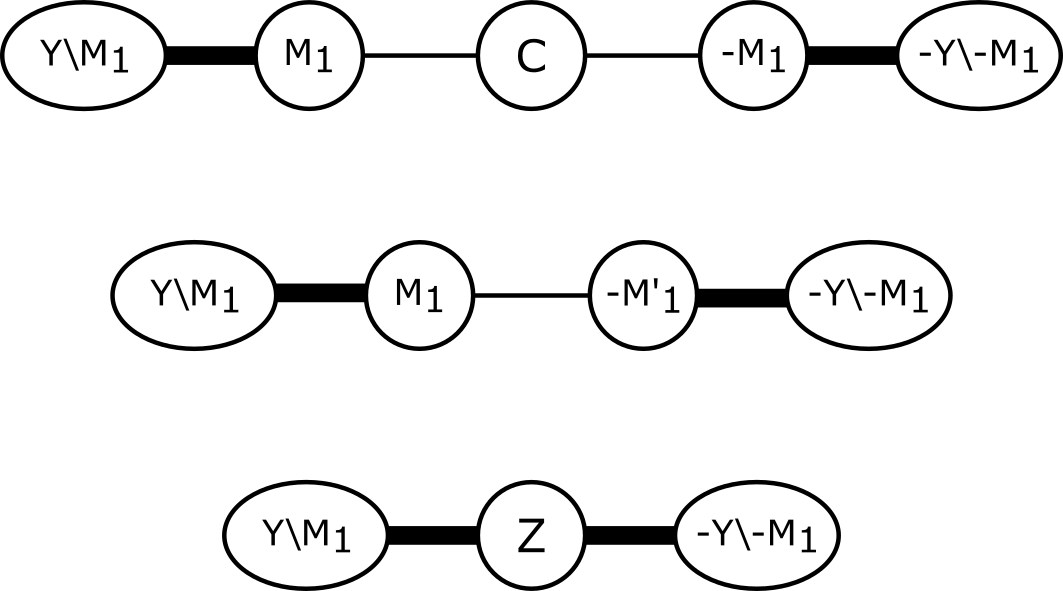}
    \caption{The pictures from top to bottom illustrate the JSJ-trees of $\tilde{X}$ in cases $i$, $i^\prime$, and $ii$, respectively. Replacing $C$ and $Z$ by $C_n$ and $Z_n$ gives the JSJ-trees of $\tilde{X}(\frac{1}{n})$, and $-M_1^\prime = C_{\pm 1} \cup -M_1$. Each circular node represents a JSJ-piece, and each elliptic node represents a union of JSJ-pieces. The thick lines represents multiple edges. 
    }
    \label{figure:JSJtrees}
    \end{figure}

\begin{lemma}{\label{JSJpiecesofX}}
Suppose $Y$ is not a solid torus. Let $\{M_j\}$ denote the JSJ-pieces of $Y$ such that $M_1$ is the one containing $T_1$ and $\{-M_j\}$ denotes those of $-Y$. Then one of the following holds:
\begin{enumerate}[i.]
    \item $M_1$ is not a Seifert fibered space whose regular fibers are meridional on $T_1$, then the JSJ-pieces of $\tilde{X}$ are $ \{M_j\} \cup \{C\} \cup \{ -M_j \}$,
    \item $M_1$ is a Seifert fibered space whose regular fibers are meridional on $T_1$ and $M_1 \neq Y$, then the JSJ-pieces of $\tilde{X}$ are $ \{M_j\} \cup \{ Z \} \cup \{ -M_j \}$ with $j \neq 1$ and where $Z$ is the Seifert fibered space $M_1 \cup C \cup -M_1$,
    \item $M_1$ is a Seifert fibered space whose regular fibers are meridional on $T_1$ and $M_1 = Y$, then $\tilde{X}$ is a Seifert fibered space.
\end{enumerate}
For $|n| \geq 2$, the statement also holds for $\tilde{X}(\frac{1}{n})$ when $C$ and $Z$ are replaced by $C_n$ and $Z_n$, respectively. When $|n|=1$, we need to replace Case $i$ with the following.
\begin{enumerate}[i.]
 \item[$i^\prime$.] $M_1$ is not a Seifert fibered space whose regular fibers are meridional on $T_1$, then the JSJ-pieces of $\tilde{X}(\pm 1)$ are $\{M_j\} \cup \{M_1\} \cup \{C_{\pm 1} \cup -M_1\} \cup \{ -M_j \}$ with $j \neq 1$.
 \end{enumerate}
\end{lemma}

See Figure \ref{figure:JSJtrees} for the JSJ-trees in cases $i, i^\prime$ and $ii$.

\subsection{Homeomorphisms of JSJ-pieces}
As any homeomorphism is isotopic to one that leaves the union of JSJ-tori invariant, it restricts to homeomorphisms on JSJ-pieces and unions of JSJ-pieces. As seen previously, we have a composing space and a cable space as JSJ-pieces,  hence it is worth to us to know about homeomorphisms between Seifert fibered spaces. We start recalling some basic facts about Seifert fibered manifolds, which are quite elementary to the experts (as a matter of fact, almost all are Seifert's original results \cite{Se33}) and can be found in many textbooks. We choose to follow \cite{Ma16}. Two Seifert fiberings of a 3-manifold $M$ are said to be \textit{isomorphic} if there is a diffeomorphism from one to another that is fiber-preserving. Then we say the Seifert fibration of $M$ is \textit{unique up to isomorphism} if any two Seifert fibration of $M$ is isomorphic. Given two distinct parameterizations, whether they describe isomorphic fibrations or not is determined as follows.

\begin{proposition}{\label{PropMoves}}
$S(g,n: (\alpha_1,\beta_1), \dots, (\alpha_k,\beta_k))$ and $S(g,n:(\alpha^\prime_1,\beta^\prime_1, \dots , (\alpha^\prime_m, \beta^\prime_m))$ describe two orientation-preservingly isomorphic Seifert fibrations if and only if they are related by a finite sequence of the following moves and their inverses:
\begin{enumerate}[1.]
    \item $(\alpha_i,\beta_i),(\alpha_{i+1},\beta_{i+1}) \mapsto (\alpha_i, \beta_i + \alpha_i),(\alpha_{i+1},\beta_{i+1} - \alpha_{i+1})$,
    \item $(\alpha_1,\beta_1), \dots, (\alpha_k,\beta_k) \mapsto (\alpha_1,\beta_1), \dots, (\alpha_k,\beta_k),(1,0)$,
    \item $(\alpha_i,\beta_i) \mapsto (\alpha_i, \beta_i + \alpha_i)$ if $n \neq 0$,
\end{enumerate}
and permutations of the pairs $(\alpha_i,\beta_i)$'s.
\end{proposition}

Let $T_i$ denote the boundary of a fibered solid torus neighborhood of the singular fiber with multiplicity $(\alpha_i,\beta_i)$. Move 1 is the result of twisting along a fibered annulus connecting the tori $T_i$ and $T_{i+1}$; this self-diffeomorphism acts on $T_i$ and $T_{i+1}$ like two opposite Dehn twists along the fiber direction and extends to an isomorphism of the fibrations. Move 2 corresponds to drilling a nonsingular fibered torus and refilling it back. Move 3 is twisting along an annulus connecting $T_i$ and a boundary component. Then a fiber-preserving homeomorphism consists of these three moves and their inverses, and permutations of the singular fibers or the boundary components. This motivates the following definition.
\begin{definition}
The Euler number of the Seifert fibration $S(g,n: (\alpha_1,\beta_1), \dots, (\alpha_k,\beta_k))$ is the rational number
\begin{align*}
    e = \sum\limits_{i=1}^k \frac{\beta_i}{\alpha_i}.
\end{align*}

\end{definition}

When $n = 0$ the Euler number, and when $n \neq 0$ the Euler number modulo $\mathbb{Z}$ are isomorphism invariants of Seifert fibrations. In fact, if there exists an orientation-preserving homeomorphism between two isomorphic Seifert fibrations of a manifold with boundary that preserves the fibrations, then the Euler number is quite powerful to determine details about the induced homeomorphism on the boundary. Thus, we have the following.

\begin{proposition}{\label{PropHatcher}}
Let $M$ be a Seifert fibered space. Suppose $S_1$ and $S_2$ are two isomorphic Seifert fibrations of $M$. If $\partial M = \emptyset$, then $e(S_1)=e(S_2)$. If $M$ has boundary, then $e(S_1) \equiv e(S_2)\pmod{1}$ and, furthermore; if $e(S_1) \neq e(S_2)$, then any orientation-preserving homeomorphism from $S_1$ to $S_2$ that is also fiber-preserving has the effect of a total number of $e(S_1)-e(S_2)$ Dehn twists along the fiber directions on the boundary.
\end{proposition}

\begin{proof}
The first statement and the first part of the second statement are obvious by Proposition \ref{PropMoves} as Moves 1 and 2 do not change the Euler number while Move 3 increases it by 1. For the last statement, let $S_1 = S(g,n: (\alpha_1,\beta_1), \dots, (\alpha_k,\beta_k))$ and $S_2 = S(g,n:(\alpha^\prime_1,\beta^\prime_1, \dots , (\alpha^\prime_m, \beta^\prime_m))$, and $\{M_1, \dots ,M_n\}$ be the connected components of $\partial M$. Suppose there exists an orientation-preserving homeomorphism $h$ from $S_1$ to $S_2$. Then there exists a bijective map $r:\{1,\dots ,n\} \rightarrow \{1, \dots , n\}$ defined by $h(M_i)=M_{r(i)}$ for $i \in \{1, \dots, n\}$. For convenience, let $h_i = (h|_{M_i})_* : H_1(M_i) \rightarrow H_1(M_{r(i)})$. We may choose a basis $(s_i,f_i)$ for $H_1(M_i)$ such that $s_i$ is determined by a fixed section of $S_1$ and $f_i$ is a fiber of $S_1$. Because $M$ has a unique Seifert fibration up to isomorphism, $h$ is fiber-preserving and hence $h(f_i)=f_{r(i)}$. Then $h(s_i) = s_{r(i)} + k_i f_{r(i)}$. Since $e(S_1) \neq e(S_2)$, by Proposotion \ref{PropMoves} some of these Dehn twists result from twisting along an annulus connecting a neighborhood of a singular fiber and a boundary component. Then for each $i\in \{1,\dots ,k\}$ there exists $k^\prime_i \in \mathbb{Z}$ such that $h$ sends $(\alpha_i,\beta_i)$ to $(\alpha_j, \beta_j - k^\prime_i \alpha_j)$ for some $j\in \{1, \dots, n\}$ (possibly with $i=j$) such that $\Sigma_{i=1}^k k^\prime_i = \Sigma_{i=1}^n k_i$ and we see that $\Sigma_{i=1}^k k^\prime_i = e(S_1)-e(S_2)$. Thus, we obtain the desired equality $\Sigma_{i=1}^n k_i = \Sigma_{i=1}^k k^\prime_i = e(S_1)-e(S_2)$.

\end{proof}
As it turns out there are only a few 3-manifolds that have non-isomorphic Seifert fibrations. The classification of Seifert fibrations up to isomorphism is completely known and is given below. We write $S^3$ and $S^2 \times S^1$ as the lens spaces $L(1,n)$ and $L(0,1)$, respectively, and $P(p,q)$ to denote the prism manifolds.

\begin{theorem}{\label{SFSclassification}}
Every Seifert fibered space has a unique Seifert fibration up to isomorphism, except the following:
\begin{itemize}
    \item $L(p,q)$ fibres over $S^2$ with $\leq$ 2 singular points in infinitely many ways,
    \item $\mathbb{D}^2 \times S^1$ fibres over $\mathbb{D}^2$ with $\leq 1$ singular points in infinitely many ways,
    \item $P(p,q) \cong S(0,0;(2,1),(2,-1),(p,q)) \cong S(-1,0;(q,p))$,
    \item $M \simtimes S^1 \cong S(0,1;(2,1),(2,-1))$,
    \item $K \simtimes S^1 \cong S(0,0;(2,1),(2,1),(2,-1),(2,-1))$.
\end{itemize}
Here $M$ and $K$ are the Mobius strip and the Klein bottle. 
\end{theorem}  

As now we know almost all Seifert fibered spaces admit a unique Seifert fibration up to isomorphism, we may further wish to know whether any homeomorphism between isomorphic Seifert fibrations is fiber-preserving or isotopic to one that is fiber-preserving. We say the Seifert fibration of $M$ is \textit{unique up to isotopy} if any homeomorphism between two Seifert fibrations of $M$ is isotopic to a fiber-preserving homeomorphism. It is determined in \cite[Corollary~3.12]{JWW01} exactly which Seifert manifolds have a unique Seifert fibration up to isotopy. Recall from \cite[p.~446]{Sc83} that there is a unique orientable Euclidean 3-manifold which admits the Seifert fibration $S(-1,0;(2,1),(2,-1))$. We denote this manifold by $M_{\mathbb{R}P^2(2,2)}$. The list in the below theorem is sharp as there exist homeomorphisms on these manifolds that do not preserve any Seifert fibrations they admit.

\begin{theorem}[\cite{JWW01}]{\label{SFSclassuptoiso}}
Let $M$ be a Seifert fibered space. Then the Seifert fibration of $M$
is unique up to isotopy if and only if $M$ is not $T^3$, $M_{\mathbb{R}P^2(2,2)}$, $T^2 \times I$, or one of the
manifolds listed in Theorem \ref{SFSclassification}.
\end{theorem}

Any JSJ-torus in a $\mathbb{Q}$-homology sphere bounds a $\mathbb{Q}$-homology solid torus on each side. Because $\tilde{X}(\frac{1}{n})$ is a $\mathbb{Q}$-homology sphere, it is also important for us to know about homeomorphisms between $\mathbb{Q}$-homology solid tori. When dealing with homeomorphisms between manifolds with torus boundary, we need to know how the slopes are mapped. There is a well-defined slope called \textit{the rational longitude} on the boundary of a $\mathbb{Q}$-homology solid torus and as seen later this slope will be essential in our argument. For more detail on the rational longitude, we refer the reader to \cite[Section~3.1]{Wa12}.

\begin{definition}
Let $M$ be a $\mathbb{Q}$-homology solid torus and $i_* :H_1(\partial M;\mathbb{Q}) \rightarrow H_1(M;\mathbb{Q})$ be the homomorphism induced by the inclusion map $i: \partial M \hookrightarrow M$. Then the rational longitude $\lambda_M$ of $M$ is defined to be the unique slope on $\partial M$ such that $i_*(\lambda_M)$ is zero. 
\end{definition}

\begin{remark}
Let $M_i$ denote a JSJ-piece of a $\mathbb{Q}$-homology solid torus $M$, then each connected component of $M \setminus M_i$ except the one containing $\partial M$ is a $\mathbb{Q}$-homology solid torus as well. Geometrically, the rational longitude $\lambda_M$ is characterized among all slopes by the property that finitely many coherently oriented parallel copies of it bounds an essential surface in $M$. Also note that for homological reasons any homeomorphism between two $\mathbb{Q}$-homology solid tori is rational longitude-preserving.
\end{remark}

In our setting, we observe the following.
\begin{lemma}{\label{FneqRL}}
The rational longitudes of $\tilde{X}, Y,$ and $-Y$ are not meridional. Furthermore, suppose there exists an incompressible torus $T$ in $\tilde{X}$ and let $N$ be the $\mathbb{Q}$-homology solid torus bounded by $T$. Then there cannot exist an essential annulus in $\tilde{X}$ such that one of its boundary lies on $T$ and has slope $\lambda_N$, and the other is meridional on $\partial \tilde{X}$.
\end{lemma}

\begin{proof}
Let $V$ denote the solid torus glued to $\tilde{X}$ to obtain $\tilde{X}(\lambda_{\tilde{X}})$. From the definition of the rational longitude and the surgery description, it is immediate that the Mayer-Vietoris sequence for the triple $(\tilde{X},V,\tilde{X}(\lambda_{\tilde{X}}))$ yields $H_2(\tilde{X}(\lambda_{\tilde{X}})) \neq 0$. If $\lambda_{\tilde{X}}$  = $\mu$, then $\tilde{X}(\lambda_{\tilde{X}})$ $ = \tilde{X}(\mu) = \Sigma(J \# -J)$. But $ H_2(\Sigma(J \# -J)) = 0$, so we reach a contradiction. Same argument works for $Y$ (and $-Y$) as $H_2(Y(\lambda_Y)) \neq 0$ and $Y(\mu_Y)\cong \Sigma(J)$. For the second statement, suppose for contradiction that such an annulus exists and call it $A$. Let $D$ denote the disk bounded by the boundary component of $A$ on $\partial \tilde{X}$ in $\tilde{X}(\mu)$. Then a regular neighborhood of $N \cup A \cup D$ is homeomorphic to $N(\lambda_N)_o$, a punctured $N(\lambda_N)$. However, again the Mayer-Vietoris sequence for the triple $(N(\lambda_N)_o, \tilde{X}(\mu)\setminus N(\lambda_N)_o, \tilde{X}(\mu))$ yields $H_2(\tilde{X}(\mu)) \neq 0$ and it is a contradiction.

\end{proof}

The following is a technical lemma that will be used in proving Theorem \ref{ThmC}.

\begin{lemma}{\label{Ltorus}}
Let $\varphi: T^2 \rightarrow T^2$ be an orientation-preserving homeomorphism. For a fixed meridian-longitude pair for $T^2$, suppose there are two slopes $\frac{p}{q}$ and $\frac{p^\prime}{q^\prime}$ that are preserved under $\varphi$. Then either $\varphi$ is isotopic to $\pm Id$ or $\frac{p}{q} = \frac{p^\prime}{q^\prime}$.
\end{lemma}

\begin{proof}
Let $\varphi_*$ be the induced isomorphism on $H_1(T^2)$. Then $\varphi_*$ sends $(p,q)$ to $\pm(p,q)$ and similarly $(p^\prime, q^\prime)$ to $\pm(p^\prime,q^\prime)$. Let's first consider the cases $\varphi_*$ fixes at least one of them. Then, by a change of basis we may assume it is  $(p^\prime,q^\prime)=(1,0)$. Let $\big(\begin{smallmatrix}
  a & b\\
  c & d
\end{smallmatrix}\big)$ $\in SL_2(\mathbb{Z})$ represent $\varphi_*$. As $\varphi_*((1,0))=(1,0)$, $a=1$ and $c=0$. Then $d=1$. If $\varphi_*((p,q))=(p,q)$, then $bq=0$. If $b=0$, then $\varphi_* = Id$, and if $q=0$, then $(p,q)=(1,0)$. If $\varphi_*((p,q))=-(p,q)$, then we find that $p=q=0$ but then $\frac{p}{q}$ does not define a slope. Now consider the case in which $\varphi_*$ flips the signs of both elements. Then $-\varphi_*$ preserves both. As handled previously we see that either $\varphi_*=-Id$ or $(p,q)=(p^\prime,q^\prime)$ and complete the proof.

\end{proof}

\subsection{Proof of Theorem \ref{ThmC}}
Now suppose $K$ is amphichiral. Then $\Sigma(K) \cong \Sigma(-K)$ as oriented manifolds, and hence there exists an orientation-preserving homeomorphism $h: \tilde{X}(\frac{1}{n}) \rightarrow \tilde{X}(-\frac{1}{n})$.

Without loss of generality, we may assume $n>0$. For simplicity, let us denote $C_n$ and $C_{-n}$ by $C_+$ and $C_-$, and similarly denote $Z_n$ and $Z_{-n}$ by $Z_+$ and $Z_-$. Also, let $T_+$ denote the JSJ-torus that cobounds $M_1$ and $ C_+ \cup -M_1$ and let $T_-$ the one that cobounds $M_1$ and $C_- \cup -M_1$. As any homeomorphism is isotopic to one that leaves the union of JSJ-tori invariant, $h$ induces a graph isomorphism on JSJ-graphs. In Cases $i$ and $ii$, as seen from Figure 4,  $C_\pm$ and $Z_\pm$ are central vertices in the JSJ-trees of $\tilde{X}(\frac{1}{n})$ and $\tilde{X}(-\frac{1}{n})$ whose removals result in two isomorphic trees. So, we see that $h$ sends $C_+$ to $C_-$, and $Z_+$ to $Z_-$. Similarly, $h$ sends $T_+$ to $T_-$ in Case $i^\prime$.

We will now show that the existence of such an $h$ would always lead to a contradiction and complete the proof of Theorem \ref{ThmC}.

\begin{lemma}
There is no orientation-preserving homeomorphism from $\tilde{X}(\frac{1}{n})$ to $\tilde{X}(-\frac{1}{n})$.
\end{lemma}

\begin{proof}
We begin by considering the following case. 

\textbf{$Y$ is a solid torus:} Recall that $Y = \Sigma(B_l,t_l)$ where $B_l = B_L \setminus \text{int }(B_L \cap B)$ and $t_l = t \cap B_l$. As $Y$ is a solid torus, $t_l$ is a $\frac{p}{q}$-tangle for some $\frac{p}{q} \in \mathbb{Q} \cup \{\frac{1}{0}\}$. Hence, $Y$ fibers as $S(0,1;(q,p))$ (and similarly, $-Y$ fibers as $S(0,1;(q,-p))$ ). Hence we also obtain $\tilde{X}$ to be the Seifert fibered space $S(0,1:(q,p),(q,-p))$. Then $\frac{1}{n}$-Dehn filling of $\tilde{X}$ gives rise to a singular fiber with coefficient $(n, 1)$ and caps off $\partial \tilde{X}$ which results in a closed Seifert fibered space $S(0,0:(q,p),(q,-p),(n,1))$. Similarly, $\tilde{X}(-\frac{1}{n})$ is also a closed Seifert fibered space $S(0,0:(q,p),(q,-p),(n,-1))$. We find $e(\tilde{X}(\frac{1}{n}))=\frac{1}{n}$ and $e(\tilde{X}(-\frac{1}{n})) = - \frac{1}{n}$. Then, by Proposition \ref{PropHatcher} the Seifert fibrations of $\tilde{X}(\frac{1}{n})$ and $\tilde{X}(-\frac{1}{n})$ are not isomorphic. Hence, by Theorem \ref{SFSclassification} $\tilde{X}(\frac{1}{n})$ (and similarly $\tilde{X}(- \frac{1}{n})$) is either a lens space, or a prism manifold, or $K \simtimes S^1$. It cannot be a prism manifold or $K \simtimes S^1$ as a prism manifold has 2-torsion in its first homology and $K \simtimes S^1$ has infinite order first homology. Then we conclude that $\tilde{X}(\frac{1}{n})$ is a lens space.  As lens spaces fiber with at most $2$ singular fibers, either $q=1$ or $n=1$. If $q=1$, the closure of $t_l$, the partial knot $J$, will be the unknot. But this contradicts our assumption. Then $n=1$, in this case we find $\tilde{X}(1)$ to be the lens space $L(q^2, 1 + qs)$ and $\tilde{X}(-1)$ to be the lens space $L(q^2,-1+qs)$ where $r,s \in \mathbb{Z}$ such that $ps + qr =1$. Then, by the classification of lens spaces, we must have either $1+qs \equiv -1 + qs \pmod{q^2}$ or $(1 + qs)(-1 +qs) \equiv 1\pmod{q^2}$. Solving both equations, we get $\pm 2 \equiv 0 \pmod{q^2}$. Then, $q = 1$ or $q=2$. The case $q=1$ is handled previously, then $q=2$. But then $|H_1(\tilde{X}(\pm 1))| = |H_1(L(4,\pm 1 + 4s))| = 4$, which is a contradiction as the order of the double branched cover of a knot is always an odd number.

This completes the proof for the case in which $Y$ is a solid torus. Hence, we are left to consider the following.

\textbf{$Y$ is not a solid torus:} In light of Lemma {\ref{JSJpiecesofX}}, we will start considering Case $ii$ as most of Case $i$ follows as a special case of it.

Case \textit{ii:} In this case $Z_+$ and $Z_-$ are JSJ-pieces of $\tilde{X}(\frac{1}{n})$ and $\tilde{X}(-\frac{1}{n})$, respectively. Then, as explained above $h(Z_+) = Z_-$. By Theorem \ref{SFSclassification}, $Z_+$ and $Z_-$ are isomorphic. By the mirror symmetry, the other singularity coefficients except $(n,\pm 1)$ come in cancelling pairs and we get $e(Z_\pm) = \pm \frac{1}{n}$. But we see from Proposition \ref{PropHatcher} that if $n>2$, $Z_+$ and $Z_-$ can not be orientation-preservingly homeomorphic and we get a contradiction. Then, $n$ is either $1$ or $2$.

Now by Theorem \ref{SFSclassuptoiso}, up to isotopy we may assume that $h|_{Z_+}$ is fiber-preserving. Let $\{A_1, \dots ,A_m\}$ be the connected components of $\tilde{X}(\frac{1}{n})\setminus Z_+$. We may choose a basis $(s_i,f_i)$ for $H_1(\partial A_i)$ such that $s_i$ is determined by a fixed section of $Z_+$ and $f_i$ is a fiber of $Z_+$. Now let $r,h_i$ and $k_i$ be as in Proposition \ref{PropHatcher}, then $\Sigma_{i=1}^m k_i = \frac{2}{n}$. For each $i \in \{1,\dots ,m\}$ there exists a least $n_i \in \mathbb{N}$ such that $r^{n_i}(i)=i$. The set $\{1, \dots ,m\}$ is partitioned into sets such that $i$ and $j$ are in the same set if $i=r^n(j)$ for some $n \in \mathbb{N}$. As $\Sigma_{i=1}^m k_i = \frac{2}{n} \neq 0$, for some set $\textbf{a}$ in this partition, $\Sigma_{i \in \textbf{a}} k_i \neq 0$. Now pick a random element $j \in \textbf{a}$ and let $\overline{h} = h_{r^{n_j-1}(j)} \circ \dots \circ h_{r(j)} \circ h_j$. Then $\overline{h}(s_j)= s_j + \Sigma_{i \in \textbf{a}} k_i f_j$ and $\overline{h}(f_j)=f_j$. As each $A_i$ is a $\mathbb{Q}$-homology solid torus, each $h_i$ is also rational longitude-preserving, so is $\overline{h}$. Then, by Lemma \ref{Ltorus} either $\overline{h}=\pm Id$ or $f_j = \lambda_{A_j}$. However, $h \neq \pm Id$ as $\Sigma_{i \in \textbf{a}} k_i \neq 0$. Then $f_j = \lambda_{A_j}$. Hence, $\lambda_{A_j}$ is parallel to $\mu$ but this contradicts Lemma \ref{FneqRL}.

Case \textit{i:} Suppose $n>1$. Then, $C_+$ and $C_-$ are JSJ-pieces of JSJ-pieces of $\tilde{X}(\frac{1}{n})$ and $\tilde{X}(-\frac{1}{n})$, respectively, and hence $h(C_+) = C_-$. We immediately see that $e(C_\pm)=\pm \frac{1}{n}$. As before, by Theorem \ref{SFSclassification} $C_+$ and $C_-$ are isomorphic, but from Proposition \ref{PropHatcher} that if $n>2$, $C_+$ and $C_-$ can not be orientation-preservingly homeomorphic and we get a contradiction. Then, we conclude that $n$ is 2. Now by Theorem \ref{SFSclassuptoiso}, up to isotopy we may again assume that $h|_{C_+}$ is fiber-preserving. Now the rest of the proof follow \textit{mutatis mutandis} from the arguments in Case $ii$ where $m=2$ (and also, $A_1 = Y$ and $A_2 = -Y$).

Case \textit{i$^\prime$ :} In this case $n=1$. Recall that $C = S(0,3) = P \times S^1$ where $P$ is the pair of pants. We choose bases $(s,f)$ and $(s^\prime,f^\prime)$ for $H_1(\partial Y)$ and $H_1(\partial(-Y))$, respectively, such that $s$ and $s^\prime$ are determined by $P \times \{pt\}$ and $f$ and $f^\prime$ are fibers of $C$. Observe that $(s,f)$ and $(s^\prime, f^\prime)$ are interchanged under $\tilde{\tau}$. Then, we can view $\tilde{X}(1) = Y \cup_\varphi -Y$ where $\varphi_*(s)=s^\prime + f^\prime$ and $\varphi_*(f)=-f^\prime$. Similarly, $\tilde{X}(-1) = Y \cup_\psi -Y$ where $\psi_*(s) = s^\prime -f^\prime$ and $\psi_*(f)=-f^\prime$. As discussed above, $h(T_+)=T_-$. There is a canonical identification of $T_+$ and $T_-$ with $\partial Y$ in both $\tilde{X}(1)$ and $\tilde{X}(-1)$, then let $h_*=\big(\begin{smallmatrix}
  a & b\\
  c & d
\end{smallmatrix}\big) \in GL_2(\mathbb{Z})$ with $det$ $h_* = \pm 1$ represent the induced homomorphism $(h|_{T_+})_*$ given in the basis $(s,f)$. Let $\frac{p}{q}$ be the rational longitude of $Y$. By the choice of the pair $(s,f)$ and $(s^\prime,f^\prime)$, the rational longitude of $-Y$ has the slope $-\frac{p}{q}$. First suppose $h(Y)=Y$. Then $det$ $h_* =1$. Since $h|_Y$ and $h|_{-Y}$ are rational longitude-preserving, $h_*$ sends $(p,q)$ to either $(p,q)$ or $-(p,q)$ and $\psi_* \circ h_* \circ \varphi^{-1}_*$ sends $(-p,q)$ to either $(-p,q)$ or $(p,-q)$. First suppose  $h_*((p,q))=(p,q)$ and $\psi_* \circ h_* \circ \varphi^{-1}_*((-p,q)) = (-p,q)$. Solving these two equations yields $p=0$, but this contradicts Lemma \ref{FneqRL}. Now suppose $h_*((p,q))=-(p,q)$ and $\psi_* \circ h_* \circ \varphi^{-1}_*((-p,q)) = (-p,q)$. Then solving these two equations together with $det$ $h_* = 1$ yields either $p=0$, which is a contradiction to Lemma \ref{FneqRL}, or $b=2, c = -\frac{(a+1)^2}{2},$ and $d=-(a+2)$ for some odd $a$. Then by induction we see that
\begin{align*}
 (h_*)^n = \Bigg[\begin{matrix} (-1)^{n+1}(na+n-1) & (-1)^{n+1}2n \\
                (-1)^n \frac{n}{2} (a+1)^2 & (-1)^n(na+n+1) \end{matrix}\Bigg],   
\end{align*}
and hence $h_*$ has infinite order, but as $Y$ is not a solid torus, this is a contradiction to Johannson's finiteness theorem \cite{Jo79}. Now observe that the remaining two situations are similar to the previous ones up to sign, so the result will not change and in both situations we still get contradictions. Then we are left to consider the case $h(Y)=C_- \cup -Y$. Now $det$ $h_* = -1$. As before $h|_Y$ and $h|_{-Y}$ are rational longitude-preserving, then $\psi_* \circ h_*$ sends $(p,q)$ to either $(-p,q)$ or $(p,-q)$, and $h_* \circ \varphi^{-1}_*$ sends $(-p,q)$ to either $(p,q)$ or $-(p,q)$. First, consider $\psi_* \circ h_*$ sends $(p,q)$ to $(-p,q)$ and $h_* \circ \varphi^{-1}_*$ sends $(-p,q)$ to  $(p,q)$. As before solving these two equations yields $p=0$, and this is a contradiction to Lemma \ref{FneqRL}. Now, consider $\psi_* \circ h_*$ sends $(p,q)$ to $(p,-q)$ and $h_* \circ \varphi^{-1}_*$ sends $(-p,q)$ to  $(p,q)$. Then solving these two equations together with $det$ $h_* = 1$ yields either $p=0$, which is a contradiction to Lemma \ref{FneqRL}, or $b=-2, c = -\frac{(a-1)^2}{2} -1,$ and $d=2-a$. Let $g$ be the twisting along an annulus connecting a neighborhood of the singular fiber of $C_-$ and the boundary component of $C_- \cup -Y$ such that $g(C_- \cup -Y) = C_+ \cup Y$. Then $f=h|_Y \circ g \circ h|_{C_+ \cup -Y}$ is a homeomorphism such that $f(Y)=Y$. By induction we find that 
\begin{align*}
 (f_*)^n = \Bigg[\begin{matrix} (-1)^n(2na-(2n-1)) & (-1)^{n+1}4n \\
                (-1)^n n (a-1)^2 & (-1)^{n+1}(2na-(2n+1)) \end{matrix}\Bigg],   
\end{align*}
and hence $f_*$ has infinite order, but again as $Y$ is not a solid torus, this is a contradiction to Johannson's finiteness theorem \cite{Jo79}. The remaining two situations are similar to the previous ones up to sign, so the result will not change and in both situations we still get contradictions.

Case \textit{iii:} Now $\tilde{X}$ is a Seifert fibered space. Then $\frac{1}{n}$-Dehn filling of $\tilde{X}$ gives rise to a singular fiber with coefficient $(n, 1)$ and caps off $\partial \tilde{X}$ which results in a closed Seifert fibered space $\tilde{X}(\frac{1}{n})$. Similarly, $\tilde{X}(-\frac{1}{n})$ is also a closed Seifert fibered space with a singular fiber with coefficient $(n,-1)$. As before, by the mirror symmetry the other singularity coefficients except $(n,\pm 1)$ come in cancelling pairs, and we get $e(\tilde{X}(\frac{1}{n}))=\frac{1}{n}$ and $e(\tilde{X}(-\frac{1}{n})) = - \frac{1}{n}$. Then, by Proposition \ref{PropHatcher} the Seifert fibrations of $\tilde{X}(\frac{1}{n})$ and $\tilde{X}(-\frac{1}{n})$ are not isomorphic. Hence, by Theorem \ref{SFSclassification} $\tilde{X}(\frac{1}{n})$ (and similarly $\tilde{X}(- \frac{1}{n})$) is either a lens space, or $S(-1,0;(q,p))$, or $K \simtimes S^1$. It cannot be $S(-1,0;(q,p))$, or $K \simtimes S^1$ as $S(-1,0;(q,p))$ has 2-torsion in its first homology and $K \simtimes S^1$ has infinite order first homology. Then we conclude that $\tilde{X}(\frac{1}{n})$ is a lens space. As lens spaces fiber with at most 2 singular fibers, $Y$ has at most one singular fiber. Thus, $Y$ is a solid torus, but this is a contradiction.

\end{proof}

\begin{proof}[Proof of Theorem \ref{ThmB}]
This follows from Theorem \ref{ThmA} and Theorem \ref{ThmC}.

\end{proof}

\bibliographystyle{alpha}
\bibliography{main}
\end{document}